\documentclass[11pt]{amsart}
\usepackage{graphicx}
\usepackage{amsfonts, amsmath, amsfonts, amssymb}
\newtheorem{thm}{Theorem}[section]
\newtheorem{cor}[thm]{Corollary}
\newtheorem{lem}[thm]{Lemma}
\newtheorem{prop}[thm]{Proposition}
\theoremstyle{definition}
\newtheorem{defn}[thm]{Definition}
\theoremstyle{remark}
\newtheorem{rem}[thm]{Remark}
\numberwithin{equation}{section} \theoremstyle{quest}
\newtheorem{quest}[]{Question}
\numberwithin{equation}{section} \theoremstyle{prob}

\numberwithin{equation}{section} \theoremstyle{answer}

\numberwithin{equation}{section}
\theoremstyle{fact}

\numberwithin{equation}{section}
\theoremstyle{facts}
\newtheorem{facts}[thm]{Facts}
\numberwithin{equation}{section}

\begin{document}

\title[Triang-mmSp]{Curvature based triangulation of metric measure spaces}%
\author{
Emil Saucan}
\address{
Department of Mathematics, Technion, Haifa, Israel}
\email{
semil@tx.technion.ac.il}

\thanks{
Research supported by the Israel Science Foundation Grant 666/06 and by European Research Council under the European Community's Seventh Framework Programme
(FP7/2007-2013) / ERC grant agreement n${\rm ^o}$ [203134].}%

\date{\today}

\begin{abstract}
We prove that a Ricci curvature based method of triangulation of compact Riemannian manifolds, due to Grove and Petersen, extends to the context of weighted Riemannian manifolds and more general metric measure spaces. In both cases the role of the lower bound on Ricci curvature is replaced by the curvature-dimension condition ${\rm CD}(K,N)$. We show also that for weighted Riemannian manifolds the triangulation can be improved to become a thick one and that, in consequence, such manifolds admit weight-sensitive quasimeromorphic mappings. An application of this last result to information manifolds is considered.

Further more, we extend to weak ${\rm CD}(K,N)$ spaces the results of Kanai regarding the discretization of manifolds, and show that the volume growth of such a space is the same as that of any of its discretizations.

\end{abstract}

\maketitle

\section{Introduction}

The existence of triangulation on geometric spaces can hardly be underestimated both in Pure and in Applied Mathematics, in particular that of certain special types (see, e.g. \cite{ca1}, \cite{cms}, \cite{tu}, \cite{pe}, \cite{s1}, \cite{S2}, \cite{Br}, and \cite{Re}, \cite{ab}, \cite{SAZ}, \cite{DLJZHYG}, respectively).

For Riemannian manifolds, a number of possible constructions exist, including those that produce special types of triangulations. Amongst them we mention, in chronological order and without any pretention of being exhaustive,  \cite{ca1}, \cite{GP}, \cite{pe}, \cite{ep+:92} (Theorem 10.3.1), \cite{s1}, \cite{Br}.

We concentrate here on the method employed in \cite{GP}. (See also \cite{Be} for a history of this approach). 
The advantage of this method is, besides its elegance, the fact that it is highly geometric in spirit, using solely the {\em intrinsic} geometric differential properties of the manifold. 
The purpose of the present note is to show that the construction devised in \cite{GP} can be extended, almost without any modifications to metric measure spaces (except, of course, the obvious necessary adaptations to the more general context).

The reason behind this is dual: On one hand there exists a feeling in the community, that, although the tools and results developed by Gromov, Lott, Villani, Sturm and others are most elegant, they lack, so far, any concrete and efficient application. We wish, therefore, to further emphasize that the notions of curvatures for metric measure spaces are highly natural by giving an extension of a classical problem and its (also classical) solution, to the new context. On the other hand, it appears that there exists a real interest in the triangulation and representation of the {\em information manifold}, that is the space of parameterized probability measures (or the {\it statistical model}) equipped with the Riemannian metric induced by the {\it Fischer information metric} onto the Euclidean sphere (see Section 4.1.3 below).

The reminder of this paper is structured as follows: As already mentioned above, we present, in Section 2, the construction of Grove and Petersen. (This and the following section are the most ``didactic'', thus we proceed rather slowly; however, afterwards the pace of the exposition will become more brisk.) Next, in Section 3, we introduce the necessary notions and results regarding curvatures of metric measure spaces. (Unfortunately, the style is rather technical, since we had to cover quite a large number of definitions and results.)  Section 4 constitutes the heart of our paper, in the sense that we show therein how to extend Grove-Petersen construction to the metric measure setting. We follow, in Sections 5, with an extension, to weak ${\rm CD}(K,N)$ spaces, of Kanai's results regarding the discretization of manifolds, and show, in particular, that the volume growth of such a space coincides with that of any of its discretizations. In a sense, this represents the second part of the present paper, related to, yet distinct, from the main triangulation problem 
considered in the previous sections. We conclude, in Section 6, with a few very brief remarks.



\section{The Grove-Petersen construction}

In the following $M^n = (M^n,g)$ denotes a closed, connected $n$-dimensional Riemannian manifold such that it has sectional curvature $k_M$ bounded from below by $k$,  ${\rm diam}M^n$ bounded from above by $D$, and ${\rm Vol}M^n$ is bounded from below by $v$.
Since only volumes of balls arguments are employed, 
one can replace the last condition by the more general one ${\rm Ric_M} \geq (n-1)k$ (see, e.g., \cite{Pet}).
In fact, this very relaxation of the conditions actually helps us formulate the more general problem we are dealing with in the sequel. 

The basic idea is to use so called {\it efficient packings}:

\begin{defn}
Let $p_1,\ldots,p_{n_0}$ be points  $\in M^n$,  satisfying  the
following conditions:
\begin{enumerate}
\item The set $\{p_1,\ldots,p_{n_0}\}$ is an $\varepsilon$-net on $M^n$, i.e. the
balls $\beta^n(p_k,\varepsilon)$, $k=1,\ldots,n_0$ cover $M^n$;
\item The balls (in the intrinsic metric of $M^n$) $\beta^n(p_k,\varepsilon/2)$ are pairwise
disjoint.
\end{enumerate}
Then the set $\{p_1,\ldots,p_{n_0}\}$ is called a {\it minimal
$\varepsilon$-net} and the packing with the balls
$\beta^n(p_k,\varepsilon/2)$, $k=1,\ldots,n_0$, is called an {\it
efficient packing}. The set $\{(k,l)\,|\,k,l = 1,\ldots,n_0\; {\rm
and}\; \beta^n(p_k,\varepsilon) \cap \beta^n(p_l,\varepsilon) \neq
\emptyset\}$ is called the {\it intersection pattern} of the minimal
$\varepsilon$-net (of the efficient packing).
\end{defn}


Efficient packings have the following important properties, which we
list below (for proofs see \cite{GP}):

\begin{lem}[\cite{GP}, Lemma 3.2]
There exists $n_1 = n_1(n,k,D)$, such that if
$\{p_1,\ldots,p_{n_0}\}$ is a minimal $\varepsilon$-net on $M^n$, then $n_0
\leq n_1$.
\end{lem}

\begin{lem}[\cite{GP}, Lemma 3.3]
There exists $n_2 = n_2(n,k,D)$, such that for any $x \in M^n$,
$\left|\{j \,|\, j = 1,\ldots,n_0\; {\rm and}\;
\beta^n(x,\varepsilon) \cap \beta^n(p_j,\varepsilon) \neq
\emptyset\}\right| \leq n_2$, for any minimal $\varepsilon$-net
$\{p_1,\ldots,p_{n_0}\}$.
\end{lem}

\begin{lem}[\cite{GP}, Lemma 3.4]
Let $M_1^n, M_2^n$, be manifolds having the same bounds $k =k_1 = k_2$
and $D = D_1  = D_2$ (see above) and let $\{p_1,\ldots,p_{n_0}\}$ and
$\{q_1,\ldots,q_{n_0}\}$ be minimal $\varepsilon$-nets with the same
intersection pattern, on $M_1^n$, $M_2^n$, respectively. Then
there exists a constant $n_3 = n_3(n,k,D,C)$, such that if
$d(p_i,p_j) < C\cdot\varepsilon$, then $d(q_i,q_j) <
n_3\cdot\varepsilon$. 
\end{lem}

This properties 
suffice to provide us with a simple and efficient, even if crude, triangulation method of closed, connected
Riemannian manifolds. Indeed, by using the properties above, one can
construct a simplicial complex
 having as vertices the centers of the balls
 $\beta^n(p_k,\varepsilon)$., as follows: Edges are connecting the centers of
 adjacent balls; further edges being added to ensure the cell
 complex obtained is triangulated to obtain a simplicial complex.
One can ensure that the triangulation will be convex and that its
simplices are convex, by choosing $\varepsilon = {\rm
ConvRad}(M^n)$, where the {\it convexity radius} ${\rm
ConvRad}(M^n)$ is defined as follows:

\begin{defn}
Let $M^n$ be a Riemannian manifold. The {\it convexity radius} of
$M^n$ is defined as $\inf\{r>0\,|\, \beta^n(x,r) \; {\rm is\;
convex},\; {\rm for\; all\;} x \in M^n\}$.
\end{defn}

This follows from the fact that $\beta^n\left(x,{\rm
ConvRad}(M^n)\right) \subset \beta^n\left(x,{\rm
InjRad}(M^n)\right)\,$, (since ${\rm ConvRad}(M^n) \geq \frac{1}{2}{\rm
InjRad}(M^n)$ -- see, e.g.  \cite{Be}). Here ${\rm InjRad}(M^n)$
denotes the {\it injectivity radius}:

\begin{defn} Let $M^n$  be a Riemannian manifold. The {\em
injectivity radius} of $M^n$ is defined as: ${\rm InjRad}(M^n) =
\inf\{{\rm Inj}(x)\,|\,x \in M^n\}$, where ${\rm Inj}(x) = \\ \sup{\{{r\,|\,
{\rm exp}_{x}|_{\mathbb{B}^n(x,r)}\; {\rm is\;
a \; diffeomorphism}}\}}$.
\end{defn}

Note that by a classical result of Cheeger (see, e.g. \cite{Be}), there is a
universal positive lower bound for ${\rm InjRad}(M)$ in terms of
$k, D$ and $v$, where $v$ is the lower bound for the volume
of $M$. It is precisely this result (and similar ones -- see also
the discussion below) that make the triangulation exposed above a
simple and practical one, at least in many cases.

\begin{rem}

Lemmas 2.2 - 2.4 above represent part of the tools\footnote{A notable part of the other tools being represented by a generalization of Cheeger's ``butterfly'' construction \cite{Chg} (see also \cite{Fuk}).}
 employed in the proof of
 main theorem of \cite{GP}, namely:

\begin{thm}[\cite{GP}, Theorem 4.1] \label{thm:GP-Main}
Let $M_1^n, M_2^n$ be two manifolds having the same upper diameter bound $D$, as well as the same lower bounds $k$ and $v$, on their curvatures and volumes, respectively. Then there exists $\varepsilon = \varepsilon(n,k,D,v)$ such that, if $M_1$ and $M_2$ have minimal packings with identical intersection patterns, then they are homotopy equivalent.
\end{thm}

Unfortunately, the condition regarding the sectional curvature bound cannot be replaced easily by a similar one regarding Ricci curvature, not even in the classical Riemannian case,\footnote{ See Berger \cite{Be} for a brief discussion on the results employing Ricci curvature bounds.} and evidently not in the more general setting adopted in this paper.
\end{rem}

We bring, for reference, the proofs of the Lemmas above, and we do this almost verbatim:

\begin{proof}[Proof of Lemma 2.2]
Let $\{p_1,\ldots,p_{n_0}\}$ be a minimal $\varepsilon$-net on $M^n$ and let $\tilde{p}$ be a point in $\widetilde{M}^n_k$ -- the $k$-space form. Then, by the classical Bishop-Gromov Theorem\footnote{It basically states that the volume of balls in a complete Riemannian manifold $(M^n,g)$ satisfying
${\rm Ric} \geq(n-1)k$, does not increase faster than the volume of balls in the model space form, more precisely that, for any $x \in M = M^n$, the function

\[\varphi(r) = \frac{{\rm Vol}B(x,r)}{\int_0^rS_K^n(t)dt}\,,\]
is nonincreasing (as function of $r$), where

\[
S_K^n(r) = \left\{
                  \begin{array}{ll}
                    \Big(\sin{\sqrt{\frac{K}{n-1}}r}\Big)^{n-1} & \mbox{if $K > 0$}\\\\
                    r^{n-1} & \mbox{if $K = 0$}\\\\
                     \Big(\sinh{\sqrt{\frac{|K|}{n-1}}r}\Big)^{n-1} & \mbox{if $K < 0$}
                  \end{array}
           \right.
\]

}

\[\frac{{\rm Vol}B(p,r)}{{\rm Vol}B(p,R)} \geq \frac{{\rm Vol}B(\tilde{p},r)}{{\rm Vol}B(\tilde{p},R)}\,, 0 < r < R\,;\]
for any $p \in M^n$.

Let $i_0$ such that ${\rm Vol}B(p_{i_0},\varepsilon/2)$ is minimal. By 2.1.(2) it follows that

\[n_0 \leq \frac{{\rm Vol}M^n}{{\rm Vol}B(p_{i_0},\varepsilon/2)} \leq \frac{{\rm Vol}B(\tilde{p},D)}{{\rm Vol}B(\tilde{p},\varepsilon/2)}\,.\]
(To obtain the last inequality, just take, in Bishop-Gromov Theorem, $R = {\rm diam}M^n \leq D$.)

The desired conclusion now follows by taking

\[n_1 = \left[\frac{{\rm Vol}B(\tilde{p},D)}{{\rm Vol}B(\tilde{p},\varepsilon/2)}\right]\,.\]
\end{proof}

\begin{proof}[Proof of Lemma 2.3]
Let $j_1,...,j_s$ be such that $B(x,\varepsilon) \cap B(p_{j_i},\varepsilon) \neq \emptyset$. Then $B(p_{j_i},\varepsilon/2) \subset B(x,5\varepsilon/2)$.

Let $k \in \{1,..., s\}$ be such that $B(p_{j_k},\varepsilon/2)$ has minimal volume. Then (as in the proof of Lemma 2.2) it follows that:

\[s \leq \frac{{\rm Vol}B(x,5\varepsilon/2)}{{\rm Vol}B(p_k,\varepsilon/2)} \leq \frac{{\rm Vol}B(p_{j_k},9\varepsilon/2)}{{\rm Vol}B(p_{j_k},\varepsilon/2)} \leq \frac{{\rm Vol}B(\tilde{p},9\varepsilon/2)}{{\rm Vol}B(\tilde{p},\varepsilon/2)}\,,\]
where $\tilde{p}$ is as in the proof of the previous lemma. But 

\[\frac{{\rm Vol}B(\tilde{p},9\varepsilon/2)}{{\rm Vol}B(\tilde{p},\varepsilon/2)} = \frac{\int_{0}^{9\varepsilon/2}S_K^n(r)dr}{\int_{0}^{\varepsilon/2}S_K^n(r)dr}\,,\]
and the function
\[h(\varepsilon) = \frac{\int_{0}^{9\varepsilon/2}S_K^n(r)dr}{\int_{0}^{\varepsilon/2}S_K^n(r)dr}\]
extends to a continuous function $\tilde{h}:[0,D] \rightarrow \mathbb{R}_+$,
(since $h(\varepsilon)  \rightarrow 0$ when $\varepsilon \rightarrow 0$).

\end{proof}

\begin{rem}
It is important to note that $n_2$ is independent of $\varepsilon$.
\end{rem}

\begin{proof}[Proof of Lemma 2.4]
 Evidently, since $d(p_i,p_j) < C \cdot \varepsilon$, it follows that $p_j \in B(p_i,C \cdot \varepsilon)$. Thus, precisely as in the proof of the previous lemma, it follows that there exists $n' = n'(C)$,

\begin{displaymath}
n'(C) = \max{\frac{\int_{0}^{(4k+1)\varepsilon/2}S_K^n(r)dr}{\int_{0}^{\varepsilon/2}S_K^n(r)dr}}\,,
\end{displaymath}
such that at most $n'$ of the balls $B(p_1,\varepsilon/2),...,B(p_{n'},\varepsilon/2)$ are included in $B(p_i,(C+\frac{1}{2})\varepsilon/2$.

Since $\{p_1,\ldots,p_{n_0}\}$ and
$\{q_1,\ldots,q_{n_0}\}$ have the same intersection pattern, it follows that $d(q_i,q_j) \leq n_3(C)$, where $n_3(C) = 2[n'(C)-1]$.

\end{proof}

%
%
%


\section{Ricci curvature of metric measure spaces}

We bring here the definitions and results necessary in the following section. However, since proofs are not needed, we omit them and send to the source (usually \cite{Vi}).\footnote{However, the reader should be aware that, while we mostly refer to \cite{Vi} for convenience, many (if not all) the results therein appeared first in \cite{LV} and/or \cite{St}.}

We begin with the basic (and motivating)  
generalization, namely:

\subsection{Riemannian manifolds}

Let $M = M^n$ be a complete, connected $n$-dimensional Riemannian manifold.

One wishes to extend results regarding Ricci curvature to the case when $M^n$ is equipped with a measure that is not $d{\rm Vol}$. 
Usually (at least in our context) such a measure is taken to be of the form

\[\nu(dx) = e^{-V(x)}{\rm Vol}(dx)\,,\]
where $V:M^n \rightarrow \mathbb{R}$, $V \in \mathcal{C}^2(\mathbb{R})$. Note also that any smooth positive probability measure can be written in this manner. Then $(M,d, \nu)$, where $d$ is the geodesic distance, is a metric measure space.

\begin{rem}
A standard measure $\nu$, in the context of Image Processing (but not only) would be the gaussian measure on $\mathbb{R}^n$:

\[\gamma^{(n)} = \frac{e^{-|x|^2}dx}{(2\pi)^{n/2}}\,.\]

However, more realistic measures can (and, indeed, should) be considered for imaging purposes -- see \cite{HM}, \cite{HLM}.
\end{rem}

To 
preserve geometric significance of the Ricci tensor, one has to modify its definition as follows:

\begin{equation} \label{eq:Ric-N-Nu}
{\rm Ric}_{N,\nu} = {\rm Ric} + \nabla^2V - \frac{\nabla V \otimes
\nabla V}{N - n}
\end{equation}

Here $\nabla V \otimes \nabla V$ is a quadratic form on $TM^n$, and $\nabla^2V$ is the Hessian matrix ${\rm Hess}$, defined as:

\[(\nabla V \otimes \nabla V)_x(v) = (\nabla V(x) \cdot v)^2\,.\]

Therefore

\[{\rm Ric}_{N,\nu}(\dot{\gamma}) = ({\rm Ric} + \nabla^2V)(\dot{\gamma}) - \frac{(\nabla V \cdot \dot{\gamma})^2}{N - n}\,.\]

Here $N$ is the so called {\em effective dimension} and is to be inputed.

\begin{rem}
\begin{enumerate}
\item If $N < n$ then ${\rm Ric}_{N,\nu} = -\infty$
\item If $N = n$ then, by convention, $0 \times \infty = 0$, therefore (\ref{eq:Ric-N-Nu}) is still defined even if $\nabla V = 0$, in particular
${\rm Ric}_{n,{\rm Vol}} = {\rm Ric}$ (since, in this case $V \equiv 0$).
\item If $N = \infty$ then ${\rm Ric}_{\infty,\nu} = {\rm Ric} + \nabla^2V$.
\end{enumerate}
\end{rem}


The Ricci curvature boundedness condition of the classical Bishop-Gromov is paralleled in the case of smooth metric measure spaces by the following (rather obvious)

\begin{defn}
$(M,d, \nu)$ satisfies the {\it curvature-dimension} estimate ${\rm CD}(K,N)$ iff there exist $K \in \mathbb{R}$ and $N \in [1,\infty]$, such that  ${\rm Ric}_{N,\nu} \geq K$ and $n \leq N$. (If $\nu = d{\rm Vol}$, then the first condition reduces to the classical ${\rm Ric} \geq K$.)
\end{defn}

\begin{rem}
Intuitively, ``$M$ has dimension $n$ but pretends to have dimension $N$. (Identity theft)''\footnote{J. Lott \cite{L}.}

The need for such a parametric dimension stems, in particular, 
from the desire 
to extend the Bishop-Gromov Theorem to metric spaces (or more precisely, to length spaces), for which no innate notion of dimension exists.
\end{rem}

\begin{rem}
For a number of equivalent conditions, see \cite{Vi}, Theorem 14.8.
\end{rem}

\begin{thm}[Generalized Bishop-Gromov Inequality] \label{thm:BG+}
Let $M$ be a Riemannian manifold equipped with a reference measure $\nu = e^{-V}{\rm Vol}$ and satisfying a curvature-dimension condition $CD(K,N)$, $K \in \mathbb{R}, 1 < N < \infty$.

Then, for any $x \in M$, the function

\[\varphi(r) = \frac{\nu\left[B(x,r)\right]}{\int_0^rS_K^N(t)dt}\,,\]
is nonincreasing (as function of $r$), where

\[
S_K^N(t) = \left\{
                  \begin{array}{ll}
                    \Big(\sin{\sqrt{\frac{K}{N-1}}t}\Big)^{N-1} & \mbox{if $K > 0$}\\\\
                    t^{N-1} & \mbox{if $K = 0$}\\\\
                     \Big(\sinh{\sqrt{\frac{|K|}{N-1}}t}\Big)^{N-1} & \mbox{if $K < 0$}
                  \end{array}
           \right.
\]

\end{thm}

\begin{proof}
See \cite{Vi}, p. 499-500.
\end{proof}




\subsection{Metric Measure Spaces}

To further extend the Bishop-Gromov Theorem, we have first to introduce a proper generalization of "Ricci curvature bounded from below". For this, quite a number of preparatory definitions are needed, that we briefly review here.

\begin{defn}
Let $(X,\mu)$ and $(Y,\nu)$ be two measure spaces. A {\it coupling} (or {\it transference (transport) plan}) of $\mu$ and $\nu$ is a measure $\pi$ on $X \times Y$ with {\it marginals} $\mu$ and $\nu$ (on $X$ and $Y$, respectively), i.e. such that, 
for all measurable sets $A \subset X$ and $B \subset Y$, the following hold: $\pi[A \times Y] = \mu[A]$ and $\pi[X \times B] = \nu[B]$.
\end{defn}

\begin{defn}
Let $(X,\mu)$ and $(Y,\nu)$ be as above and let $c = c(x,y)$ be a (positive) cost function on $X \times Y$. Consider the {\it Monge-Kantorovich minimization problem}:

\begin{equation}
\inf\int_{X \times Y}c(x,y)d\pi(x,y)\,,
\end{equation}
where the infimum is taken over all the transport plans. The transport plans attaining the infimum are called {\it optimal transport (transference) plans}.
\end{defn}

Before we can proceed, we must recall the following definition and facts:

\begin{defn}
Let $(X,d)$ be a Polish space, and let $P(X)$ denote the set of Borel probability measures on $X$.
Then the {\it Wasserstein distance} (of order $2$) on $P(X)$ is defined as

\begin{equation}
W_2(\mu,\nu) = \left(\inf\int_X{d((x,y)^2d\pi(x,y)}\right)^\frac{1}{2}\,,
\end{equation}
where the infimum is taken over all the transference plans between $\mu$ and $\nu$.

\end{defn}

\begin{defn}
The {\it Wasserstein space} $P_2(X)$ is defined as

\begin{equation}
P_2(X) = \Big\{\mu \in P(X) \,\big|\, \int_X{d(x_0,x)^2\mu(dx)} < \infty \big\}\,,
\end{equation}
where $x_0 \in X$ is an arbitrary point.

The definition above does not depend upon the choice of $x_0$ and $W_2$ is a metric on $P_2(X)$.
\end{defn}

\begin{facts}
Let $(X,d)$ and $P_2(X)$ be as in the definition above. Then
\begin{enumerate}
\item
$P_2(X)$ is a Polish space.
\item
If $X$ is compact, then $P_2(X)$ is also compact.
\end{enumerate}
\end{facts}

\begin{defn}
Let $(X,d)$ be a compact, geodesic, Polish space and let $\Gamma = \{\gamma:[0,1] \rightarrow X\,|\, \gamma {\rm \; a \; minimal \; geodesic}\}$, and denote by $e_t:\Gamma \rightarrow X$ the (continuous) {\em evaluation map}, $e_t:(\gamma) = \gamma(t)$. Let $E:\Gamma \rightarrow X \times X$ be defined as $E(\gamma) = (e_0(\gamma),e_1(\gamma))$. A {\it dynamical transference plan} is a pair $(\pi,\Pi)$, where $\pi$ is a transference plan and $\Pi$ is a Borel measure, such that $E_\#\Pi = \pi$. $(\pi,\Pi)$ is called {\it optimal} if $\pi$ is optimal.
\end{defn}

\begin{defn}
Let $\Pi$ be an optimal dynamical transference plan. Then the one-parameter family $\{\mu_t\}_{t \in [0,1]}, \mu_t = (e_t)_\#\Pi$ is called a {\it displacement interpolation}
\end{defn}

We can now quote the following result (\cite{Vi}, Theorem 7.21 and Corollary 7.22), connecting the geometry of the Wasserstein space to
classical mass transport:

\begin{prop}
Any displacement interpolation is a Wasserstein geodesic, and conversely, any Wasserstein geodesic is obtained as a displacement interpolation from an optimal displacement interpolation.
\end{prop}

\begin{defn}
Given $N \in [1,\infty]$, the {\it displacement convexity class} $\mathcal{DC}_N$ is defined as the set of convex, continuous functions $U:\mathbb{R}_+ \rightarrow \mathbb{R}$, $U \in \mathcal{C}^2(\mathbb{R}_+ \setminus \{0\})$, such that $U(0) = 0$ and such that

\[
\frac{rU'(r) - U(r)}{r^{1-1/N}}
\]
is nondecreasing (as a function of $r$).
%
\end{defn}

\begin{rem}
For equivalent defining conditions for the class $\mathcal{DC}_N$ see \cite{Vi}, Definition 17.1.
\end{rem}

\begin{defn}
Let $(X,d, \nu)$ be a a locally compact metric measure space, such that the measure $\nu$ is locally finite, and let $U$ be a
continuous, convex function $U:\mathbb{R}_+ \rightarrow \mathbb{R}$, $U \in \mathcal{C}^2(\mathbb{R}_+ \setminus \{0\})$, such that $U(0) = 0$. Consider a measure  $\mu$ on $X$, having compact support, and let $\mu = \rho\nu + \mu_s$ be its Lebesgue decomposition into absolutely continuous and singular parts.

Then we define the (integral) functional $U_\nu$ (with {\it nonlinearity} $U$ and {\it reference measure} $\nu$) by

\begin{equation}
U_\nu = \int_X U\big(\rho(x)\big)\nu(dx) + U'(\infty)\mu_s[X]\,.
\end{equation}

Moreover, if $\{\pi(dy|x)\}_{x \in X}$ is a family of probability measures on $X$ and if $\beta:U \times U \rightarrow (0,\infty]$ is a measurable function, we define an (integral) functional $U_{\pi,\nu}^\beta$ (with {\it nonlinearity} $U$, {\it reference measure} $\nu$, {\it coupling} $\pi$ and {\it distortion coefficient} $\beta$) by:

\begin{equation}
U_{\pi,\nu}^\beta = \int_{U \times U} U\left(\frac{\rho(x)}{\beta(x,y)}\right)\beta(x,y)\pi(dy|x)\nu(dx) + U'(\infty)\mu_s[X]\,.
\end{equation}

\end{defn}

Usually (e.g. in the definition of weak ${\rm CD}(K,N)$ spaces) $\beta$ is taken to be the {\it reference distortion coefficients}:

\begin{defn}
Let $x,y$ be two points in a metric space $(X,d)$, and consider the numbers $K \in $, $N \in [1,\infty]$ and  $t \in [0,1]$. We define the {\it reference distortion coefficients} $\beta^{(K,N)}_{t}(x,y)$ as follows:

\begin{enumerate}

\item If $t \in (0,1]$ and $1 < N < \infty$, then

\begin{equation}
\beta^{(K,N)}_{t}(x,y) = \left\{
                  \begin{array}{ll}
                    + \infty & \mbox{if $K > 0$ and $\alpha > \pi$}\,,\\\\
                    \Big(\frac{\sin{(t\alpha)}}{t\sin{\alpha}}\Big)^{N-1} & \mbox{if $K > 0$ and $\alpha \in [0,\pi]$}\,,\\\\
                    1 & \mbox{if $K = 0$}\,,\\\\
                     \Big(\frac{\sinh{(t\alpha)}}{t\sinh{\alpha}}\Big)^{N-1} & \mbox{if $K < 0$}\,;
                  \end{array}
           \right.
\end{equation}
where
\begin{equation}
\alpha = \sqrt{\frac{|K|}{N-1}}d(x,y)\,.
\end{equation}

\item In the limit cases $N \rightarrow 1$ and $N \rightarrow \infty$, define

\begin{equation}
\beta^{(K,1)}_{t}(x,y) = \left\{
                  \begin{array}{ll}
                   +\infty & \mbox{if $K > 0$}\,,\\\\
                   1 & \mbox{if $K \leq 0$}\,;
                  \end{array}
           \right.
\end{equation}

and

\begin{equation}
\beta^{(K,\infty)}_{t}(x,y) = e^{\frac{K}{6}(1-t^2)d(x,y)\,.}
\end{equation}

\item If $t = 0$, then

\begin{equation}
\beta^{(K,N)}_{0}(x,y) = 1\,.
\end{equation}

\end{enumerate}

\end{defn}

\begin{rem}
If $X$ is the model space for ${\rm CD}(K,N)$ (see \cite{Vi} p. 387), then $\beta^{(K,N)}$ is the distortion coefficient on $X$.
\end{rem}


We can now bring the definition we are interested in:

\begin{defn}
Let $(X,d,\nu)$ be a locally compact, complete, $\sigma$-finite metric measure geodesic space, and let $K \in \mathbb{R}, N \in [1,\infty]$.
We say that $(X,d,\nu)$ satisfies a {\em weak ${\rm CD}(K,N)$ condition} (or that it is a {\em weak ${\rm CD}(K,N)$ space}) iff for any two probability measures $\mu_0, \mu_1$ with compact supports ${\rm Supp}\,\mu_1, {\rm Supp}\,\mu_2 \subset {\rm Supp}\,\nu$, there exist a {\em displacement interpolation} ${\mu_t}_{0 \leq t \leq 1}$ and an associated optimal coupling $\pi$ of $\mu_0,\mu_1$ such that, for all $U \in \mathcal{DC}_N$, and for all $t \in [0,1]$, the following holds:

\begin{equation}
U_\nu(\mu_t) \leq (1-t)\,U_{\pi,\nu}^{\beta^{(K,N)}_{1-t}}(\mu_0) + t\,U_{\tilde{\pi},\nu}^{\beta^{(K,N)}_{t}}(\mu_1)
\end{equation}
(Here we denote ${\tilde{\pi}} = S_\#\pi$, where $S(x,y) = (y,x)$.)
\end{defn}

\begin{rem}
In fact, the geodesicity condition is somewhat superfluous, since a locally compact, complete metric space is geodesic (see, e.g. \cite{He}, 9.14).
\end{rem}

Fortunately, 
a version of the Bishop-Gromov Inequality also holds in general metric measure spaces; more precisely we have

\begin{thm}[Bishop-Gromov Inequality for Metric Measure Spaces] \label{thm:BG++}
Let $(X,d, \nu)$ be a weak ${\rm CD}(K,N)$ space, $N < \infty$, and let $x_0 \in {\rm Supp}\,\nu$.
Then, for any $r > 0$, $\nu[B(x_0),r] = \nu[B[x_0],r]$. Moreover,

\begin{equation}
\frac{\nu[B[x_0],r]}{\int_0^r{S_K^N(t)}dt}
\end{equation}
is a nonincreasing function of $r$.
\end{thm}
%

\begin{proof}
See \cite{Vi}, p. 806.
\end{proof}

For us, the following corollary of this extension of the Bishop-Gromov Theorem is very important, since it allows for a {\it lower} bound on the number of balls in a efficient packing:

\begin{cor}[Measure of small balls in  weak $CD(K,N)$ spaces] \label{cor:small-balls}
Let $(X,d, \nu)$ be a weak ${\rm CD}(K,N)$ space, $N \in [1,\infty)$, and let $z \in {\rm Supp}\nu$. Then, for any $R > 0$, there exists $c = c(K,N,R)$ such that, if $B(x,r) \subset B(z,R)$, then:

\begin{equation}
\nu[B(x,r)] \geq \big(c\nu[B(z,R)]\big)r^N\,.
\end{equation}
\end{cor}

Since, by \cite{Vi}, Theorem 29.9, any smooth ${\rm CD}(K,N)$ metric measure space is also a weak ${\rm CD}(K,N)$ space, 
it follows that a lower bound on the number of balls in an efficient packing also exists for smooth metric measure spaces.

Moreover, an expected doubling property also holds:

\begin{cor}[Weak $CD(K,N)$ spaces are locally doubling]
Let $(X,d, \nu)$ be a weak ${\rm CD}(K,N)$ space, $N \in [1,\infty)$, and let $z \in {\rm Supp}\nu$.
Then, for any fixed ball $B(z,R) \subset X$, there exists $C = C(K,N,R)$ such that, for all $r \in (0,R)$, $\nu[B(z,2r)] \leq C\nu[B(z,r)]$.
In particular, if ${\rm diam}X \leq D$, then $X$ is $C$-doubling, where $C = C(K,N,D)$.
\end{cor}


\section{Triangulating metric measure spaces}

\subsection{Smooth metric measure spaces}

\subsubsection{The basic construction}

The lemmas regarding  efficient packings translate to the context of smooth metric measure spaces with little, if any, modifications:

\begin{lem}[Lemma 2.2 on manifolds with density] \label{lem:2.2+}
Let $(M^n,d,\nu)$, $\nu = e^{-V}{\rm Vol}$, be a compact(closed) smooth metric measure space  such that $Ric_{N,\nu} \geq K$, for some $K \in \mathbb{R}, 1 < N < \infty$, and such that ${\rm diam}M^n \leq D$.
Then there exists $n_1 = n_1(N,K,D)$, \
such that if
$\{p_1,\ldots,p_{n_0}\}$ is a minimal $\varepsilon$-net on $M^n$, then $n_0
\leq n_1$.
\end{lem}

\begin{proof}
Since the only essential ingredient of the proof of Lemma 2.2 is the Bishop-Gromov volume comparison theorem, and since by Theorem \ref{thm:BG+} its analogue also holds for smooth metric measure spaces, the proof follows immediately precisely on the same lines as that of Lemma 2.4.
\end{proof}

\begin{lem}[Lemma 2.3 on manifolds with density] \label{lem:2.3+}
Let $(M^n,d,\nu)$ be as in Lemma \ref{lem:2.2+}.
There exists $n_2 = n_2(N,K,D)$, such that, for any $x \in M^n$,
$\left|\{j \,|\,\newline j = 1,\ldots,n_0\; {\rm and}\;
\beta^n(x,\varepsilon) \cap \beta^n(p_j,\varepsilon) \neq
\emptyset\}\right| \leq n_2$, for any minimal $\varepsilon$-net
$\{p_1,\ldots,p_{n_0}\}$.
\end{lem}

\begin{proof}
Again, as in the proof of the previous lemma, the only relatively new ingredient is the generalized Bishop-Gromov Theorem, which as we have seen, holds for smooth metric measure spaces.
\end{proof}

\begin{lem}[Lemma 2.4 on manifolds with density] \label{lem:2.4+}
Let $(M_1^n,d_1,\nu_1)$ and $(M_2^n,d_2,\nu_2)$ be as in Lemma \ref{lem:2.2+}.
and let $\{p_1,\ldots,p_{n_0}\}$ and $\{q_1,\ldots,q_{n_0}\}$ be minimal $\varepsilon$-nets with the same
intersection pattern, on $M_1^n$, $M_2^n$, respectively. Then
there exists a constant $n_3 = n_3(N,K,D,C)$, such that if
$d_1(p_i,p_j) < C\cdot\varepsilon$, then $d_2(q_i,q_j) <
n_3\cdot\varepsilon$.
\end{lem}

\begin{proof}
Idem, the proof follows along the lines of the original proof (of Lemma 2.4), by force of Theorem \ref{thm:BG+}.
\end{proof}

\begin{rem}
Note that the constants $n_1, n_2, n_3$, as functions of the effective dimension $N$, rather than the topological one $n$, are only very weekly dependent (in general) on the geometry of the manifold $M^n$, via the inequality condition $n \leq N$. While they still will guarantee the existence of minimal $\varepsilon$-nets with the required properties, the metric density of this nets will be different from the purely geometric one given by the classical Grove-Petersen construction, hence so will be the shape (or ``thickness'', see Definition 4.8 below) of the simplices of the resulting triangulation.
\end{rem}

The basic triangulation process 
now carries to the context of smooth metric measure spaces without any modification. The convexity of the triangulation follows exactly like in the classical case, since the injectivity and convexity radius are solely functions of the metric $d$ and not of the weighted volume $\nu$.%

\begin{rem}
As we have already noted in Remark 2.7, the lower bound on the sectional curvature is essential for the proof of the Homotopy Theorem \ref{thm:GP-Main}. Therefore, one cannot formulate a similar theorem for the weighted manifolds, without imposing the additional constraint on sectional curvature. The original argument of Grove and Petersen still holds for the Riemannian structure, while measure has no other role than determining the existence and metric density of the triangulation vertices. In consequence, we can only obtain a rather weak result, that we bring here (for completeness' sake) only as corollary, namely:

\begin{cor}[Theorem \ref{thm:GP-Main} for smooth metric measure spaces] \label{thm:GP-Main-SmoothMMSp}
Let $(M_1^n,d_1,\nu_1)$, $(M_2^n,d_2,\nu_2)$, $\nu_i = e^{-V_i}d{\rm Vol}$, $V_i \in \mathcal{C}^2(\mathbb{R})$, $i = 1,2$ be smooth, compact metric measure spaces satisfying ${\rm CD}(K,N)$ for some $K \in \mathbb{R}$ and $1 < N < \infty$, and such that ${\rm diam}M^n_i < D$, ${\rm Vol}M_i^n < v$, $i=1,2$ and, moreover, having the same lower bound $k$ on their sectional curvatures.
Then there exists $\varepsilon = \varepsilon(N,K,k,D,v)$ such that, if $M_1^n, M_2^n$ have minimal packings with identical intersection patterns, they are homotopy equivalent.
\end{cor}

\end{rem}


\subsubsection{Thick triangulation and quasimeromorphic mappings}

We can strengthen the simple result above, to render a ``geometrically nice'' triangulation, namely we can formulate the following

\begin{prop} \label{prop:thick-triang}
Any smooth, compact metric measure  space $(M^n,d,\nu)$ satisfying ${\rm CD}(K,N)$ admits a $\varphi^*$-thick triangulation, where $\varphi^* = \varphi^*(n,d,\nu)$.
\end{prop}

Recall that {\it thick} (or {\it fat}) triangulations are defined as follows:

\begin{defn} Let $\tau \subset \mathbb{R}^n$ ; $0 \leq k \leq n$ be a $k$-dimensional simplex.
The {\it thickness}  $\varphi$ of $\tau$ is defined as being:
\begin{equation}
\varphi = \varphi(\tau) = \hspace{-0.3cm}\inf_{\hspace{0.4cm}\sigma
< \tau
\raisebox{-0.25cm}{\hspace{-0.9cm}\mbox{\scriptsize${\rm dim}\,\sigma=j$}}}\!\!\frac{\rm Vol_j(\sigma)}{\rm diam^{j}\,\sigma}\;.
\end{equation}
The infimum is taken over all the faces of $\tau$, $\sigma < \tau$,
and ${\rm Vol}_{j}(\sigma)$ and ${\rm diam}\,\sigma$ stand for the Euclidian
$j$-volume and the diameter of $\sigma$ respectively. (If
${\rm dim}\,\sigma = 0$, then ${\rm Vol}_{j}(\sigma) = 1$, by convention.)
 A simplex $\tau$ is $\varphi_0${\it-thick}, for some $\varphi_0 > 0$,
if $\varphi(\tau) \geq \varphi_0$. A triangulation (of a submanifold
of $\mathbb{R}^n$) $\mathcal{T} = \{ \sigma_i \}_{i\in \bf I}$ is
$\varphi_0${\it-thick} if all its simplices are $\varphi_0$-thick. A
triangulation $\mathcal{T} = \{ \sigma_i \}_{i\in \bf I }$ is {\it
thick} if there exists $\varphi_0 \geq 0$ such that all its
simplices are $\varphi_0${\it-thick}.
\end{defn}

(The definition above is the one introduced in \cite{cms}. For some equivalent definitions of thickness, see \cite{ca1},
 \cite{mun}, \cite{pe}, \cite{tu}, \cite{Whi}.)

\begin{rem}
By \cite{cms}, pp. 411-412, a triangulation is thick iff the dihedral angles (in any dimension) of all the simplices of the triangulation are bounded away from zero.
\end{rem}

\begin{proof} 
Since the geometry of the manifold is not affected by the presence of the measure $\nu$, the $\delta$-{\it transversality} and $\varepsilon$-{\it moves} arguments of Cheeger et al. apply unchanged -- see \cite{cms} for the lengthy technical details.

Moreover, the resulting triangulation is ``$\nu$-sensitive'', insomuch as the metric density of the vertices, hence the
thickness $\varphi^*$ of the triangulation's simplices is a function of the measure $\nu$ (as shown by Lemmas \ref{lem:2.2+}--\ref{lem:2.4+}  above).
\end{proof}

%

\begin{rem}
As the proof above shows, the role of the measure in determining the thickness of the triangulation is somewhat marginal. Therefore, it would be useful to modify the thickness condition so that it will reflect (and adapt to) the presence of the measure $\nu$. (This may prove to be even more relevant in the case of weak ${\rm CD}(K,N)$ spaces.)
%
%
However, if the new angles $\measuredangle_{K,N}$ depend continuously on the standard ones, then, given that $\nu = e^{-V}{\rm Vol}$ and $V \in \mathcal{C}^2(\mathbb{R})$, the continuity and compactness arguments involved in the proof of the classical case also hold for generalized (dihedral) angles, 
a generalization 
of Proposition \ref{prop:thick-triang} will follow immediately
(as will, in consequence, the corresponding one of Corollary \ref{cor:qm-mmsp} below).
\end{rem}

\begin{rem}
For Riemannian manifolds, on can construct thick  triangulations, for non-compact manifolds with\footnote{ at least for many cases} or without boundary,  (see, e.g. \cite{pe}, \cite{s1}, \cite{S3}).
However, for weighted manifolds this is far from obvious, because there exists no a priori information regarding the properties of the restriction of the measure $\nu$ to the $(n-1)$-dimensional manifolds required in the process of the triangulation.
\end{rem}

Before proceeding further, we have to remind the reader the definition of
{\it quasimeromorphic mappings}:

\begin{defn}
Let $M^n, N^n$ be oriented, Riemannian $n$-manifolds.
\begin{enumerate}
\item $f:M^n \rightarrow N^n$ is called {\it quasiregular} ($qr$) iff
\begin{enumerate}
\item $f$ is locally Lipschitz (and thus differentiable a.e.);
\\ \hspace*{-0.8cm} and
\item \(0 < |f'(x)|^n \leq KJ_f(x)\), for any \(x \in M^n\);
\end{enumerate}
\;where $f'(x)$ denotes the formal derivative of $f$ at $x$,
$|f'(x)| = \sup
\raisebox{-0.25cm}{\mbox{\hspace{-0.75cm}\tiny$|h|=1$}}|f'(x)h|$,
and where $J_f(x) = detf'(x)$;
\item {\it quasimeromorphic} ($qm$) iff $N^n = \mathbb{S}^n$,
where $\mathbb{S}^n$ is usually identified with
$\widehat{\mathbb{R}^n} = \mathbb{R}^n \cup \{\infty\}$ endowed with the {\it spherical metric}. 
\end{enumerate}
The smallest number $K$ that satisfies condition (b) above is called
the {\it outer dilatation} of \nolinebreak[4]$f$.
\end{defn}

 Using the arguments of, say, \cite{s1} we obtain the following

\begin{cor} \label{cor:qm-mmsp}
Any smooth, compact metric measure  space $(M^n,d,\nu)$ satisfying ${\rm CD}(K,N)$ admits a non-constant quasimeromorphic mapping $f:M^n
\rightarrow \mathbb{S}^n$. 
\end{cor}


\subsubsection{An Application -- The Information Manifold}
The method of triangulation and quasimeromorphic mapping of weighted Riemannian manifolds presented above represents a generalization of a result  classical in {\it information geometry}.  To be somewhat more concrete  we present it briefly here (for more details see, e.g. \cite{AN}, \cite{GMRMT}).

Let  $A$ be a finite set, let $f_i(x)$, $i = 1,2$ be bounded distributions on $A$,  and let $p_i(x) = \frac{f_i(x)}{\sum_A{f_i(x)}}$, viewed as probability densities on $A$.
The {\it relative information} between $p_1$ and $p_2$ (or the {\it Kullback-Leibler divergence}) is defined as

\[KL(p_1\|p_2) = \sum_Ap_1\log{\left(\frac{p_1}{p_2}\right)}\,,\]
represents a generally accepted measure of the divergence between the two given probabilities, but, unfortunately, it fails to be a metric. However, it induces a Riemannian metric on $P(A)$ -- the manifold of probability densities on $A$, namely the {\it Fisher information metric}:

\begin{equation}
g_{{\rm Fischer},p}(\Delta) = KL(p,p+\Delta) = \sum_A{\frac{\Delta(x)^2}{p(x)}}\,,
\end{equation}
where $p \in P(A)$ is given and $\Delta$ represents an infinitesimal perturbation. (Here, to ensure that $p+\Delta$ will also be a probability density,  the following normalization is applied: 
$\int_A\Delta(x)dx = 0$.)

\begin{rem}
It turns out (see \cite{AN}) that the Fisher information can be written as Riemannian metric (in standard form) in the following form: $g_{{\rm Fischer},\cdot} = (g_{ij})$, where

\begin{equation}
g_{ij} = E_p\left(\frac{\partial \mathbf{l}}{\partial \theta^i},\frac{\partial \mathbf{l}}{\partial \theta^j}\right)\,,
\end{equation}
where $\theta^1,\ldots,\theta^k, k = |A|$ represent the coordinates on $\sigma_0$, $\mathbf{l} = \log{p}$ is the so called {\it log-likelihood} and $E_p(fg)$ denotes the {\it expectation} of $fg$, $E_p(fg) = \int{fgdp}$.
\end{rem}

The correspondence $p(x) \mapsto  u(x) = 2\sqrt{p(x)}$ maps the probability simplex
$\sigma_0 = \{p(x)\,|\,x \in A, p(x) > 0, \sum_Ap(x) = 1\}$, onto the first orthant of the sphere $S = \sum_A{u(x)^2} = 4$. This mapping preserves the geometry, in the sense that the geodesic distance between $p,q \in \sigma_0)$, measured in the Fisher metric, equal the spherical distance between their images (under the mapping above). Moreover, geodesics are mapped to great circles. (For details and more geometric insight regarding the probability simplex and the map above, see \cite{AN}.)

To summarize: The quasimeromorphic mapping of a weighted Riemannian manifold $(M^n,d,\nu)$ onto the $n$-dimensional unit sphere $\mathbb{S}^n$, represents a generalization of the considerents above in two manners:

(a) It allows for the mapping with controlled and bounded distortion (i.e. $qm$) of a more general class of Riemannian manifolds (with arbitrary metrics) endowed with a variety of (probability) measures, and not just of the standard statistical model;

(b) It permits the reduction to the study of the (geometry of the ) standard simplex in $\mathbb{S}^n$, of the geometry of the whole information manifold, and not just of the probability simplex.  (As far as application of such generalizations are concerned we mention here only those related to {\it signal} and {\it image processing} (see e.g.  \cite{Ge}, \cite{GMRMT}, respectively), and information theory (see, e.g. \cite{AN}).)

To be sure, a number of minor technical details need to be considered, namely the fact that we allow for multidimensional spheres, i.e. we consider also multivariate distributions, as opposed to the simpler notation adopted above; and the use of the standard simplex (first orthant) in the unit sphere, rather than in the sphere of radius 4, which is trivial, giving that dilation is a (elementary) conformal mapping.



\begin{rem}
As we have seen above, bounds on the injectivity (hence convexity) radius exist in terms of curvature (in conjunction with diameter and volume), due to what represent, by now, classical results. This conducts us to formulate the following question: 

\begin{quest}
Given a smooth metric measure space, is it possible to find bounds on the injectivity radius in terms of the ${\rm CD}(K,N)$ condition?
\end{quest}

In fact, it should be even possible to find bounds on the {\em conjugate radius} (see, e.g. \cite{Be}), since conjugate points appear when the volume density vanishes (see \cite{Pet}, Section 2.4.4) and this is a function of the determinant Jacobian, which is controlled by the Ricci curvature (see \cite{Pet}, Sections 2.4.2 and 6.3.1, \cite{Vi}).

\end{rem}

%



\subsection{Weak ${\rm CD}(K,N)$ spaces}

Using  Theorem \ref{thm:BG++}, the proofs of the basic lemmas, for metric measure spaces, follow immediately. We do, however, bring below their statements, for the sake of completeness:


\begin{lem}[Lemma 2.2 on weak ${\rm CD}(K,N)$ spaces] \label{lem:2.2++}
Let $(X,d,\nu)$ be a compact weak ${\rm CD}(K,N)$ space, $N < \infty$, such that ${\rm Supp}\nu = X$ and such that ${\rm diam}X \leq D$.
Then there exists $n_1 = n_1(K,N,D)$, such that if $\{p_1,\ldots,p_{n_0}\}$ is a minimal $\varepsilon$-net in $X$, then $n_0 \leq n_1$.
\end{lem}

\begin{rem}
Note that, since $N < \infty$, the condition ${\rm Supp}\nu = X$ imposes no real restriction on $X$ (see \cite{Vi}, Theorem 30.2 and Remark 30.3).
\end{rem}

\begin{lem}[Lemma 2.3 on weak $CD(K,N)$ spaces] \label{lem:2.3++}
Let $(X,d,\nu)$ be a compact weak ${\rm CD}(K,N)$ space, $N < \infty$, such that ${\rm Supp}\nu = X$ and such that ${\rm diam}X \leq D$.
Then there exists $n_2 = n_2(N,K,D)$, such that, for any $x \in M^n$,
$\left|\{j \,|\,\newline j = 1,\ldots,n_0\; {\rm and}\;
\beta^n(x,\varepsilon) \cap \beta^n(p_j,\varepsilon) \neq
\emptyset\}\right| \leq n_2$, for any minimal $\varepsilon$-net
$\{p_1,\ldots,p_{n_0}\}$.
\end{lem}

\begin{lem}[Lemma 2.4 on weak $CD(K,N)$ spaces] \label{lem:2.4++}
Let $(M_1^n,d_1,\nu_1)$ and $(M_2^n,d_2,\nu_2)$ be as in Lemma \ref{lem:2.2++}.
and let $\{p_1,\ldots,p_{n_0}\}$ and $\{q_1,\ldots,q_{n_0}\}$ be minimal $\varepsilon$-nets with the same
intersection pattern, on $M_1^n$, $M_2^n$, respectively. Then
there exists a constant $n_3 = n_3(N,K,D,C)$, such that if
$d_1(p_i,p_j) < C\cdot\varepsilon$, then $d_2(q_i,q_j) <
n_3\cdot\varepsilon$.
\end{lem}

The existence of the triangulation follows now immediately, since, by definition, weak ${\rm CD}(K,N)$ spaces are geodesic. Moreover, in {\it nonbranching} spaces, the geodesics connection two vertices of the triangulation are unique a.e. (see \cite{Vi}, Theorem 30.17). Recall that a geodesic metric  space $X$ is called nonbranching iff any two geodesics $\gamma_1, \gamma_2: [0,t] \rightarrow X$ that coincide on a subinterval $[0,t_0], 0 < t_0 < t$, coincide on $[0,t]$.

However, it is hard to ensure the convexity of the triangulation.
Sadly, many weak ${\rm CD}(K,N)$ spaces of interest fail to be locally convex. 
Fortunately, local convexity does hold for an important class of metric measure spaces: Indeed, by \cite{Ptr},  ${\rm Alex}[K] \subset {\rm CD}((m-1)K,m)$, where ${\rm Alex}[K]$ denotes the class of $m$-dimensional {\it Alexandrov spaces} with curvature $\geq K$, equipped with the volume measure.\footnote{In fact, Petrunin provides the full details of the proof only for the case $K=0$.} But, by \cite{PP}, Lemma 4.3, any point in Alexandrov space of curvature $\geq K$ has a compact neighbourhood which is also an Alexandrov space of curvature $\geq K$.

\begin{rem}
The result of Perelman and Petrunin is, in fact, a bit stronger, in the sense that the diameter of the said compact neighbourhood is specified. On the other hand, their proof uses the rather involved (even if by now standard) tool of Gromov-Hausdorff convergence (of Alexandrov spaces), amongst others.

An alternative, more intuitive (at least in dimension 2) proof can be given, using the notion of {\it Wald-Berestovskii  curvature} (see \cite{b-m:70}, \cite{pla:96}). However, since we do not want to further encumber the reader with more definitions, and since Alexandrov spaces represent by now a widely  accepted tool in Geometry, that generates important research on its own, we shall bring the full details of this proof, and of the whole construction, elsewhere.
\end{rem}

%
%
%
%
%
%
%
%
%
%
%


\begin{rem}
Since the balls $\beta^n(p_k,\varepsilon)$ cover $M^n$, any sequence $\mathcal{P}(\varepsilon_m)$, of efficient packings such that $\varepsilon_m \rightarrow 0$ when $m \rightarrow \infty$, generates a {\it discretization} $(X_m,d,\nu_m)$, in the sense of \cite{BS}: Consider the Dirichlet (Voronoi) cell complex (tesselation) $\mathfrak{C}_m = \{C(p_{m,k})_k\}$  of centers $p_{k,m}$ and atomic masses $\nu_m(p_{m,k}) = \nu[C(p_{m,k})]$. Then taking $X_m = \{p_{m,k}\}$, $d$ the original metric of $X$ and $\nu_m$ as defined, provides us with the said discretization.

It follows, by \cite{BS}, Theorem 4.1, if ${\rm Vol}(M^n) < \infty$, the sequence $(X_m,d,\nu_m)$ converges in the $W_2$  (see \cite{St}, \cite{Vi}) to a metric measure space and, moreover, if $(X,d,\nu)$ is a weak ${\rm CD}(K,N)$ space, then, for small enough $\varepsilon$, so will be $(X_m,d,\nu_m)$,  but only in a generalized (``{\it rough}'') sense. (For details regarding the precise definition of rough curvature bounds and the proof of this and other related results, see \cite{BS}.)
\end{rem}

%

%
%
%
%


%
%
%


\section{Discretizations and volume growth rate}


The compactness condition imposed in the previous sections stems from the need for finding estimates in Lemmas 2.2-2.4 (and their respective generalizations) and as a basic requirement for Theorem 2.8 to hold. However, the basic geometric method employed for obtaining efficient packings holds for noncompact manifolds, as well (given a lower bound on the injectivity radius).

Therefore, we discard in this section the compactness restriction, and concentrate, following Kanai \cite{Ka} (see also \cite{Ch}) on those properties of $\varepsilon$-nets that hold also on unbounded manifolds, and mainly on the volume growth rate.

Amongst these properties, the most basic is the one contained in condition (2) of Definition 2.1, that is that if $p_i, p_j \in \mathcal{N}, i\neq j$, then $d(p_i,p_j) \leq 2\varepsilon$. We also impose the additional condition that the set $\mathcal{N}$, satisfying the condition above, be maximal with respect to inclusion.

Following Kanai \cite{Ka}, we call the graph $G(\mathcal{N})$ obtained as in Section 2 above (i.e the 1-skeleton of the simplicial complex constructed therein) a {\it discretization} of $X$, with {\it separation} $\varepsilon$ (and {\it covering radius} $\varepsilon$) (or a $\varepsilon$-separated net). Further more, we say that $G(\mathcal{N})$ has {\it bounded geometry} iff there exists $\rho_0 > 0$, such that $\rho(p) \leq \rho_0$, for any vertex $p \in \mathcal{N}$, where $\rho(p)$ denotes the degree of $p$ (i.e. the number of neighbours of $p$).

\begin{rem}
Note that Kanai's discretization is different from the one considered in \cite{BS}. Evidently, there exists a close connection between the two approaches: Instead of the combinatorial length on the edges of the graph $G(\mathcal{N})$, consider the length of the geodesics (between the adjacent vertices). If one restricts himself to smooth metric measure spaces, then this graph can actually be embedded in $M^n$ (see remark after Definition 2.5). The Voronoi cell complex construction follows, of course, without any change. (One can adopt even a semi-discrete approach: Using combinatorial lengths for the edges of the graph and atomic measures (weights) for the vertices (i.e the nodes of the graph) equal to the measure of the corresponding Voronoi cell.)

Moreover, the graph $G$ considered above, together with the geodesic metric $d_g$ induced by the metric of $M^n$, represents a geodesic metric space and, as such, it possess, at each vertex $v$, a (not necessarily unique) {\it metric curvature}\footnote{See \cite{b-m:70}.} $\kappa(v)$ (e.g. the Wald-Berestovskii curvature mentioned above). Therefore, a natural question arises in connection with the discretization discussed in Remark 4.21:\footnote{ See also \cite{BS}, Example 4.4 and Section 5.}

\begin{quest}
What is the relation (for $\varepsilon$ small enough) between $\kappa_v$ above, and the rough curvature bound of the corresponding space $(X_m,d,\nu_m)$ from Remark 4.21?
\end{quest}

\end{rem}

\begin{rem}
Note that in \cite{Ka} a more general covering radius is considered. For the cohesiveness of the paper we have used $\varepsilon$ as covering radius. However, the proofs of all results using covering radius hold for any $R$.
\end{rem}

We start (presenting the results) with the following lemma:

\begin{lem}
Let $(X,d,\nu)$ be a (weak) ${\rm CD}(K,N)$ space, $K \leq 0, N < \infty$, and let $\mathcal{N}$ be a  $\varepsilon$-separated net. Then

\begin{enumerate}
\item
\[|\mathcal{N} \cap B(x,r)| \leq \frac{\int_0^{2r+\varepsilon/2}{S_K^N(t)}dt}{\int_0^{\varepsilon/2}{S_K^N(t)}dt}\,,\]

\item
\[\rho(p) \leq \frac{\int_0^{4r+\varepsilon/2}{S_K^N(t)}dt}{\int_0^{\varepsilon/2}{S_K^N(t)}dt}\,;\]
\end{enumerate}
for any $x \in X$ and $r > 0$.
\end{lem}

\begin{proof}
By Lemma \ref{lem:2.3++} and Corollary \ref{cor:small-balls}, $|\mathcal{N} \cap B(x,r)|$  is finite. The precise bounds (1) and (2) are obtained in the course of the proofs of \ref{lem:2.4++} and \ref{lem:2.3++}, (or rather in their classical versions) respectively.
\end{proof}

\begin{defn}[Rough isometry]
Let $(X,d)$ and $(Y,\delta)$ be two metric spaces, and let $f:X \rightarrow Y$ (not necessarily continuous).
$f$ is called a {\em rough isometry} iff

\begin{enumerate}
\item There exist $a \geq 1$ and $b > 0$, such that
\[\frac{1}{a}d(x_1,x_2) - b \leq \delta(f(x_1),f(x_2)) \leq ad(x_1,x_2) + b\,,\]
%
\item there exists $\varepsilon_1$ such that
\[\bigcup_{x \in X}{B(f(x),\varepsilon_1 > 0)}= Y\,;\]
(that is $f$ is $\varepsilon_1$-{\it full}.)
\end{enumerate}

\end{defn}

\begin{rem}
\begin{enumerate}
\item
Rough isometry represents an equivalence relation.
\item
If ${\rm diam}(X), {\rm diam}(Y)$ are finite, then $X$, $Y$ are roughly isometric.
\end{enumerate}
\end{rem}

The basic result of this section is the following generalization of \cite{Ka} Lemma 2.5 (see also \cite{Ch}, Theorem 4.9).

\begin{thm}
Let $(X,d,\nu)$ be a weak ${\rm CD}(K,N)$ space and let $G$ be a discretization of $X$. Then $(X,d)$ and $(G,\mathbf{d})$, where $\mathbf{d}$ is the combinatorial metric, are roughly isometric.
\end{thm}

The proof closely follows the one for the classical case, except for a necessary adaptation:

\begin{proof}
By the construction of $\mathcal{N}$ and $G(\mathcal{N})$, the following evidently holds:

\begin{equation} \label{eq:lowerbound}
d(p_1,p_2) \leq 2\varepsilon \mathbf{d}(p_1,p_2),\; p_1,p_2 \in \mathcal{N}\,,
\end{equation}
which gives the required lower bound (with $a = 1/2\varepsilon$ and without $b$) without any curvature constraint.

To prove the second inequality, one has, however, to assume that $(X,d,\nu)$ is a weak ${\rm CD}(K,N)$.

Then, given $p_1,p_2 \in \mathcal{N}$, chose a minimal geodesic $\gamma \subset G(\mathcal{N})$ with ends $p_1$ and $p_2$. Put $T_\gamma = \{y \in X\,|\,d(y,\gamma) < \varepsilon\}$ and denote $\mathcal{N}_\gamma = \mathcal{N} \cap T_\gamma$. Then there exists a path in $G(\mathcal{N})$ connecting $p_1$ and $p_2$ and contained in $\mathcal{N}_\gamma$ (see \cite{Ch}, p. 196). Therefore

\[\mathbf{d}(p_1,p_2) \leq |\mathcal{N}_\gamma|\,.\]

Let $m$ be such that $(m-1)\varepsilon \leq d(p_1,p_2) < \varepsilon$, and let $p_1=q_0,\ldots,q_m=p_2$ equidistant points on $\gamma$. Then

\[d(q_{i-1},q_i) = \frac{d(p_1,p_2)}{m} < \varepsilon\,.\]

From Lemma 5.2 it follows that

\[|\mathcal{N}_\gamma| \cap B(q_i,2\varepsilon) \leq C = C(K,N,\varepsilon))\,.\]

Moreover,

\[\mathcal{N}_\gamma \subseteq \bigcup_{i=0}^{m}B(q_i,2\varepsilon)\,.\]

We obtain, after a few easy manipulations that

\[|\mathcal{N}_\gamma| \leq \frac{C_0}{\varepsilon}d(p_1,p_2) + C_1\,,\]
where $C_0,C_1$ are constants.

\end{proof}

\begin{rem}
Holopainen \cite{Ho} has shown that if one only requires that $d(p_1,p_2) \leq 3\varepsilon$, instead of $\leq 2\varepsilon$, as in Kanai's original construction, then the theorem above can be proved by purely metric methods, without any curvature constraint. However, we have included here Kanai's version, for its Ricci curvature ``flavor''. Also, as already noted above, the construction is, in this case, geometrically natural.
\end{rem}

By Remark 5.3 (2) above, immediately follows

\begin{cor}
Any two discretizations of a weak ${\rm CD}(K,N)$ space are roughly isometric.
\end{cor}

Before proceeding further, we need to remind the reader the following definition:

\begin{defn}[Volume growth]
For $x \in X$ and $r > 0$ we denote the ``volume'' growth function by:
\begin{equation}
\mathcal{V}(x,r) = \nu[B(x,r)]\,.
\end{equation}

We say that $X$ has {\it exponential (volume) growth} iff
\begin{equation}
\lim_{r \rightarrow \infty}{\sup{\frac{\mathcal{V}(x,r)}{r}}} > 0\,,
\end{equation}

and {\it polynomial (volume) growth} iff there exists $k >0$ such that
\begin{equation}
\mathcal{V}(x,r) \leq C\cdot r^k\,,
\end{equation}
($C = {\rm const.}$) for all sufficiently large $r$.
\end{defn}

For the vertices of $G(\mathcal{N})$, there exist two natural measures: the {\it counting measure} $d\iota$: $d\iota(\mathcal{M}) = |\mathcal{M}|$,  for any $\mathcal{M} \subseteq \mathcal{N}$; and the {\it ``volume'' measure} $d\mathbf{V}$: $d\mathbf{V}(p) = \rho(p)d\iota(p)$. However, in our context, one can substitute  $d\iota$ for $d\mathbf{V}$ (and, indeed, for any measure $\mu$ absolutely continuous with respect to $d\iota$) -- see \cite{Ch}, page 197 -- so we shall work with the counting measure, because its intuitive simplicity.

\begin{lem}[Roughly isometric graphs have identical growth rate]
Let $G$ and $\Gamma$ be connected, roughly isometric graphs with bounded geometry.
Then $G$ has polynomial (exponential) growth iff $G$ has polynomial (exponential) growth.
\end{lem}

\begin{proof}
See \cite{Ka}, p. 399. 
\end{proof}

\begin{thm}[Weak ${\rm CD}(K,N)$ spaces have the same growth as their discretizations]
Let $(X,d,\nu)$ be a weak ${\rm CD}(K,N)$ space, $K \leq 0, N < \infty$, satisfying the following non-collapsing condition:

{\rm ($\ast$)} There exist $r_0, \mathcal{V}_0 > 0$ such that $\mathcal{V}(x,r_0) \geq \mathcal{V}_0$, for all $x \in X$.

Let $G$ be a discretization of $X$. Then $X$ has polynomial (exponential) volume growth iff $G$ has polynomial (exponential) volume growth.
\end{thm}

As in the case of the proof of Theorem 5.5, the proof closely mimics (modulo the necessary modifications) the one of the original result of Kanai, as is given in \cite{Ch}, pp. 198-199.

\begin{proof}
First note that, for any $r > 0$, there exists $c(r) > 0$ such that

\[\mathcal{V}(x,r) \geq c(r)\,,\]
for any $x \in X$.

Indeed, if $r \geq r_0$ then $c(r) = \mathcal{V}_0$. If $r < r_0$, then by Theorem \ref{thm:BG++}, we have

\[\mathcal{V}(x,r) \geq \frac{\int_0^{r}{S_K^N(t)}dt}{\int_0^{r_0}{S_K^N(t)}dt}\mathcal{V}(x,r_0) \geq \frac{\int_0^{r}{S_K^N(t)}dt}{\int_0^{r_0}{S_K^N(t)}dt}\mathcal{V}_0\,.\]

Now, if $x \in G$ and $y \in B(x,r)$, by the maximality of the vertex set of $G$, it follows that there exists $p \in \mathcal{N} \cap B(y,\varepsilon)$, hence $d(x,p) < r + \varepsilon$, and it follows that

\[B(x,r) \subseteq \bigcup_{p \in \mathcal{N} \cap B(x,r+\varepsilon)}B(p,\varepsilon)\,.\]

Therefore we have

\[\mathcal{V}(x,r) \leq \int_0^{r}{S_K^N(t)}dt|\mathcal{N} \cap(x,r+\varepsilon)|\,.\]

But $G$ is roughly isometric to $X$, by Theorem 5.5, that is $\mathbf{d}(p_1,p_2) \leq ad(p_1,p_2) + b$. Therefore,

\[\mathcal{V}(x,r) \leq c_1 |\mathcal{B}(x,ad(p_1,p_2) + b)|\,,\]
where $\mathcal{B}$ denotes the ball in the combinatorial metric of $G$ and $c_1$ is a constant.

This concludes the ``only if'' part of the proof.

Conversely, for any $x \in G$ and $\rho > 0$, we have

\[c_2|G \cap B(x,\rho)| \leq \sum_{p \in G \cap B(x,\rho)}{V(p,\varepsilon/2)} \leq V(x,\varepsilon/2 + \rho)\,,\]
where $c_2 = c_2(\varepsilon/2)$ is a constant.

But (\ref{eq:lowerbound}) implies that $\beta(x,\rho) \subseteq B(x,2\varepsilon\rho)$, therefore

\[c_2|\beta(x,\rho)| \leq V(x,2\varepsilon + 2\varepsilon\rho)\,,\]
which, combined with the previous inequality concludes the ``if'' part of the proof.

\end{proof}


\begin{rem}
In \cite{Ch}, {\rm ($\ast$)} appears as a condition in the statement of the Theorem and we followed this approach. For smooth manifolds, this represents, however, a consequence of negative Ricci curvature. More precisely, ${\rm Vol}B(x,r) \geq V_0r^n$, where $V_0 = V_0(n)$, for any $r < \rm InjRad(M^n)/2$. (It appears that this is a result due to Croke \cite{Cr}.) This fact further emphasize the need (already underlined in Remark 4.17 and Question 1) for finding curvature bounds for the injectivity (conjugacy) radius in metric measure spaces.
\end{rem}

\begin{cor}
Let $X_1$, $X_2$ be weak ${\rm CD}(K,N)$ spaces, $K \leq 0, N < \infty$, satisfying condition {\rm ($\ast$)} above. Then, if $X_1,X_2$ are roughly isometric, then they have the same volume growth type.
\end{cor}

\begin{proof}
Let $G_i$ be a discretization of $X_i$, $i=1,2$. Then, by Theorem 5.4, $G_i$ is roughly isometric to $X_i$, $i=1,2$. Since rough isometry is an equivalence relation, it follows, that $G_1$ and $G_2$ are roughly isometric, hence they have the same growth rate, by Lemma 5.7. Moreover, by Theorem 5.8, $G_i$ has the growth rate of $X_i$.
\end{proof}

\begin{rem}
By \cite{Vi}, Theorem 29.9 the results above hold, of course, for smooth metric measure spaces as well.
\end{rem}

\section{Final Remarks}

We have shown that both the triangulation method (and, in consequence, the existence of quasimeromorphic mappings) and the discretization results extend easily 
to a wide class of spaces,  provided that they satisfy a (generalized) Gromov-Bishop Theorem.

Of course, there are also the hidden assumptions that  the base space $X$ is ``rich'' enough and the generalized curvature is sufficiently ``smooth''. In the absence of these conditions, it is not certainly that a Gromov-Bishop type theorem can be obtained. Indeed, for the discretization due to J. Ollivier \cite{O}, it is not even clear what should be the ``reference volume'' (and even if such a volume exists). Also, unfortunately, this appears to be the case for the rough curvature bounds  \cite{B}.


\subsection*{Acknowledgment}

The author would like to thank Anca-Iuliana Bonciocat for introducing him to the fascinating field of Ricci curvature for metric measure spaces, for answering his numerous questions, for assuring him that ``things should work'' and for her helpful corrections and suggestions, and Shahar Mendelson for his stimulating questions and for his help that made this paper possible. Thanks are also due to Allen Tannenbaum for bringing to his attention the information geometry aspect and for his encouragement.



\end{document}